\documentclass[a4paper,12pt]{amsart}

\usepackage{amsmath}
\usepackage{amsthm}
\usepackage{amssymb}
\usepackage{stmaryrd}
\usepackage{MnSymbol}
\usepackage{bbm}
\usepackage{bm}
\usepackage[top=30truemm, bottom=30truemm, left=25truemm, right=25truemm]{geometry}

\numberwithin{equation}{section}

\newtheorem{theorem}{Theorem}[section]
\newtheorem{lemma}[theorem]{Lemma}

\newtheorem{proposition}[theorem]{Proposition}

\theoremstyle{definition}

\renewcommand{\Re}{\mathrm{Re}\hspace{1pt}}
\renewcommand{\Im}{\mathrm{Im}\hspace{1pt}}

\title[Joint universality for the Riemann zeta-function]{Joint Universality for the Riemann zeta-function with general shifts}
\author[K. Nakai]{Keita Nakai}
\date{}

\begin{document}

\begin{abstract}
Laurin\v{c}ikas proposed that whether the universality theorem for the Riemann zeta-function shifted by an exponential function holds or not. 
This problem is solved for more general shifts by Andersson et al. 
In this paper, we generalize their result for the joint universality theorem for the Riemann zeta-function with general shifts using a different approach from their method. 

\end{abstract}

\maketitle

\section{Introduction and the statement of main result}

\subsection{Background}

Let $s = \sigma + it$ be a complex variable. 
The Riemann zeta-function $\zeta(s)$ is defined by the infinite series 
$\sum_{n=1}^{\infty} n^{-s}$ 
in the half plane $\sigma > 1$, 
and can be continued meromorphically to the whole plane $\mathbb{C}$. 
In 1975, Voronin~\cite{Vo} proved the universality property of the Riemann zeta-function. 
The modern statement of the universality theorem is as follows. 

\begin{theorem} [The Universality Theorem]
Let $\mathcal{K}$ be a compact set in the strip $1/2 < \sigma < 1$ with connected complement, and let $f(s)$ be a non-vanishing continuous function on $\mathcal{K}$ that is analytic in the interior of $\mathcal{K}$. Then, for any $\varepsilon > 0$ 
\[
\liminf_{T \to \infty} \frac{1}{T} \mathrm{meas} \left\{\tau \in [0, T] :  \sup_{s \in \mathcal{K}} |\zeta(s + i\tau) -f(s)| < \varepsilon \right\} > 0,
\]
where $\mathrm{meas}$ denotes the 1-dimensional Lebesgue measure. 
\end{theorem}

\subsection{Universality theorems with shifts}

Kaczorowski, Laurin\v{c}ikas and Steuding~\cite{KLS} introduced the Shifts Universality Principle, which states that  
if the joint universality holds for $F_1(s), \dots$, $F_n(s)$ in the strip $\{s \in \mathbb{C} : \sigma_1 < \sigma < \sigma_2 \}$, then the joint universality holds for $F_1(s + \lambda_1), \dots, F_n(s + \lambda_n)$, where $\lambda_1, \dots, \lambda_n$ satisfy 
\begin{itemize}
    \item $\max_{1 \le j \le n} \Re(\lambda_j) - \min_{1 \le j \le n} \Re(\lambda_j) < \sigma_2 - \sigma_1$, 
    \item $K_k \cap K_j = \emptyset$ for $1 \le j  < k \le n$, where $K_j = \{s + \lambda_j : s \in K \}$ and $K \subset \{s \in \mathbb{C} : \sigma_1 < \sigma < \sigma_2 \}$.
\end{itemize} 
For example, the Dirichlet $L$-functions satisfy the Shift Universality Principle in the strip $1/2 < \sigma < 1$. 

We consider some general shifts for the Riemann zeta-function or $L$-functions. 
In other words, for real valued functions $\gamma_1, \dots, \gamma_r$, we are interested in whether the universality~(\ref{joint universality}) for zeta or $L$-functions $L_j$ holds for any $\varepsilon > 0$ or not? 
\begin{equation}
\liminf_{T \to \infty} \frac{1}{T}\mathrm{meas}\left\{\tau \in [0, T] :  \max_{1 \le j \le r}\sup_{s \in \mathcal{K}} |L_j(s + i\gamma_j(\tau)) -f(s)| < \varepsilon \right\} > 0
\label{joint universality}
\end{equation}

The Shifts Universality Principle gives a positive answer to this question when $\gamma_j(\tau) = \tau - i\lambda_j$. 
Further results regarding a class of functions for which the universality~(\ref{joint universality}) holds have been deduced by Nakamura~\cite{Na09}. It was the first result of extending the joint universality for Dirichlet $L$-functions with shifts $\gamma_j(\tau) = a_j \tau$ with real numbers $a_1, \dots, a_r$ linearly independent over $\mathbb{Q}$.
Moreover, when $r = 2$, further results were obtained in References \cite{Na09}, \cite{Pa09}, and \cite{Pa16}. 
Furthermore, Pa\'{n}kowski~\cite{Pa18} proved that the universality (\ref{joint universality}) holds when $\gamma_j(\tau) = \alpha_j \tau^{a_j} (\log{\tau})^{b_j}$ with $\alpha_j, a_j, b_j$ satisfying certain assumptions. 
Laurin\v{c}ikas and \v{S}iau\v{c}i\={u}nas~\cite{LS} showed a positive answer with shifts $\gamma_j(\tau) = t_{\tau}^{\alpha_j}$, where $t_\tau$ is the Gram function and $\alpha_1, \dots, \alpha_r$ are distinct fixed positive numbers. 

Besides explicitly determined shifts, in References \cite{LMS19}, \cite{La21Rama}, \cite{La21Math} and \cite{La22}, several variances of universality theorems for the Riemann zeta-function have been proven, for which the shifts $\gamma$ satisfy $\gamma'(2\tau)/\gamma'(\tau) \ll 1$ and further restriction. 
For example, in Reference~\cite{La22}, the joint universality for the Riemann zeta-function is proven for the shifts $\gamma_1(\tau) = \tau, \gamma_2(\tau) = \tau^2 + \tau + 1, \dots, \gamma_r(\tau) = \tau^r + \tau^{r-1} + \dots +1$. 
However, when $\gamma_1(\tau) = e^{a_1\tau}, \dots, \gamma_r(\tau) = e^{a_r\tau}$, one cannot prove the joint universality for the Riemann zeta-function using this method. 
This observation provoked Laurin\v{c}ikas to ask that when $r =1$ and $\gamma(\tau) = e^{\tau}$, whether the universality theorem for the Riemann zeta-function holds or not. 
Recently, this question has been solved in Reference~\cite{An} for more general shifts,  
whose approach includes changing of variables for the Lebesgue integral and an application of Voronin’s universality theorem. 
Thus, it is difficult to generalize their method to the joint universality for the Riemann zeta-function with general shifts. 
In this paper, we generalize this universality to the joint universality using an approach different from Reference~\cite{An}. 

\subsection{Main result}
In this subsection, we introduce our results on generalizing the Laurin\v{c}ikas problem. 

Let $\mathcal{F}$ be a set of functions $\gamma$ satisfying following assumptions. 
 
\begin{enumerate}

\item[(F1)] There exist positive real numbers $T_0 > 0$ and $T_1 > 2$ such that $\gamma : [T_0, \infty) \to [T_1, \infty)$ strictly increases. Furthermore, $\gamma(T) \to \infty$ as $T \to \infty$.
\item[(F2)] Every $\gamma$ has continuous derivative and its derivative $\gamma'$ is monotonic in $[T_0, \infty)$.
\item[(F3)] There exists a positive $\alpha$ satisfying $\alpha \gamma(\tau) \le \tau \gamma'(\tau)$ for all $\tau \ge T_0$. 

\end{enumerate}

Furthermore, we say the tuple ($\gamma_1, \dots, \gamma_r)$ with $\gamma_1, \dots, \gamma_r \in \mathcal{F}$ to be admissible if it satisfies that 
for any $(c_1, \dots ,c_r) \in \mathbb{R}^r \setminus \{\bar{0}\}$, there exist $m = m(c_1, \dots, c_r) > 0$ and $T = T(c_1, \dots, c_r) > 0$ such that $c_1 \gamma'_1(\tau) + \dots + c_r \gamma'_r(\tau)$ is monotonic and 
\[
c_1 \gamma'_1(\tau) + \dots + c_r \gamma'_r(\tau) > m\ or\  < -m
\]
for all $\tau > T$. 
A way of constructing and examples of admissible tuples can be found in \ref{ap}. 
Then the main theorem of the present paper is as follows. 

\begin{theorem} \label{main}
Let $(\gamma_1, \dots, \gamma_r)$ be admissible. 
Let $\mathcal{K}$ be a compact set in the strip $1/2 < \sigma < 1$ with connected complement, and let $f_1(s), \dots, f_r(s)$ be a non-vanishing continuous function on $\mathcal{K}$ that is analytic in the interior of $\mathcal{K}$. Then, for any $\varepsilon > 0$, 
\[
\liminf_{T \to \infty} \frac{1}{T} \mathrm{meas} \left\{\tau \in [T, 2T] :  \max_{1 \le j \le r} \sup_{s \in \mathcal{K}} |\zeta(s + i\gamma_j({\tau})) - f_j(s)| < \varepsilon \right\} > 0.
\]

\end{theorem}

\section{Preliminaries}

In this section, we introduce notations. 
First, for $\gamma_1, \dots \gamma_r \in \mathcal{F}$, we define 
\[
\underline{\zeta}(s + i\underline{\gamma}(\tau)) = (\zeta(s + i\gamma_1(\tau)), \dots, \zeta(s + i\gamma_r(\tau))).
\] 

We fix a compact subset $\mathcal{K}$ satisfying the assumptions of Theorem~\ref{main}. 
We put $\sigma_0 = (1/2 + \min_{s \in \mathcal{K}} \Re(s))/2$, $\sigma_1 = (\sigma_0 + \min_{s \in \mathcal{K}} \Re(s))/2$ and $\sigma_2 = (1 + \max_{s \in \mathcal{K}} \Re(s))/2$.
Then, we define the rectangle region $\mathcal{R}$ by 
\begin{equation} \label{definition of R}
\mathcal{R} = (\sigma_1,\ \sigma_2) \times i \left( \min_{s \in \mathcal{K}} \Im(s) - 1/2,\ \max_{s \in \mathcal{K}}\Im(s) + 1/2 \right).
\end{equation}

Then there exists a sequence of compact subsets $K_\ell$ of $\mathcal{R}$, $\ell=1, 2, \dots$ with the properties: 
\begin{itemize} 
 \item $\mathcal{R} = \bigcup_{\ell=1}^{\infty} K_\ell$, 
 \item $K_\ell \subset K_{\ell+1}$ for any $\ell \in \mathbb{N}$,
 \item For all compact subset $K$ of $\mathcal{R}$, there exists $\ell \in \mathbb{N}$ such that $K \subset K_\ell$ 
\end{itemize}
(See \cite[Chapter VII, 1.2 Proposition]{Co}). 
Let $\mathcal{H}(\mathcal{R})$ be the set of all holomorphic functions on $\mathcal{R}$. 
Now, for $g_1, g_2 \in \mathcal{H}(\mathcal{R})$, let $d_\ell(g_1, g_2) = \sup_{s \in K_\ell} |g_1(s) - g_2(s) |$ and put  
\[
d(g_1, g_2) = \sum_{\ell=1}^{\infty} \frac{1}{2^\ell} \frac{d_\ell(g_1, g_2)}{1 + d_\ell(g_1, g_2)}.
\]
Then $\mathcal{H}(\mathcal{R})$ is a metric space where $d$ is a metric on $\mathcal{H}(\mathcal{R})$,
and in particular, the metric space $\mathcal{H}(\mathcal{R})$ is a topological space with a topology of compact convergence. 
Furthermore we define $\underline{d}(\underline{g}_1, \underline{g}_2) = \max_{1 \le j \le r} d(g_{1j}, g_{2j})$
for $\underline{g}_k = (g_{k1}, \dots, g_{kr}) \in \mathcal{H}^r(\mathcal{R})$ for $k = 1, 2$. 
Then $\underline{d}$ is a metric on $\mathcal{H}^r(\mathcal{R})$. 

We write $\mathcal{B}(T)$ for the set of all Borel subsets of $T$ which is a topological space. 
Let $S^1 = \{ s \in \mathbb{C} : |s| = 1 \}$. 
For any prime $p$, we put $S_p = S^1$ and  $\Omega = \prod_p S_p$. 
Furthermore we define $\Omega^r = \Omega_1 \times \dots \times\Omega_r$, where $\Omega_j = \Omega$ for all $j = 1, \dots, r$. 
Then there exists the probability Haar measure $\mathbf{m}^r$ on $(\Omega^r, \mathcal{B}(\Omega^r))$. 
We note that $\mathbf{m}^r$ is written by $\mathbf{m}^r = \otimes_{j=i}^{r} \mathbf{m}_j$, where $\mathbf{m}_j = \otimes_{p} \mathbf{m}_p$ is the probability Haar measure on $(\Omega_j, \mathcal{B}(\Omega_j))$ and $\mathbf{m}_p$ is the probability Haar measure on $(S_p, \mathcal{B}(S_p))$. 
We define the mean $\mathbb{E}^{\mathbf{m}^r}[f] = \int_{\Omega^r} f \, d\mathbf{m}^r$.

Let $\omega_j(p)$ be the projection of $\omega_j \in \Omega_j$ to the coordinate space $S_p$ and let $\omega = (\omega_1, \dots, \omega_r)$. 
For $\omega_j \in \Omega$, we put $\omega_j(1) :=1$, 
\[
\omega_j(n) := \prod_p \omega_j(p)^{\nu(n ; p)}, 
\]
where $\nu(n ; p)$ is the exponent of the prime $p$ in the prime factorization of $n$.
Here, we define the $\mathcal{H}(\mathcal{R})$-valued random elements
\begin{align*}
\zeta(s, \omega_j) &= \prod_{p} \left(1 - \frac{\omega_j(p)}{p^s} \right)^{-1} \\
                 &= \sum_{n=1}^{\infty} \frac{\omega_j(n)}{n^s}
\end{align*}
and define $\mathcal{H}^r(\mathcal{R})$-valued random elements
\[
\underline{\zeta}(s, \omega) = (\zeta(s, \omega_1), \dots, \zeta(s, \omega_r)).
\]
We define probability measures on $(\mathcal{H}^r(\mathcal{R}), \mathcal{B}(\mathcal{H}^r(\mathcal{R})))$ by 
\begin{equation*}
\mathcal{Q}_T(A) = \frac{1}{T} \mathrm{meas} \left\{\tau \in [T, 2T] : \underline{\zeta}(s + i\underline{\gamma}(\tau)) \in A \right\}, 
\end{equation*}
\begin{equation*}
\mathcal{Q}(A) = \mathbf{m}^r \left\{\omega \in \Omega^r : \underline{\zeta}(s, \omega) \in A \right\}
\end{equation*} 
for $A \in \mathcal{B}(\mathcal{H}^r(\mathcal{R}))$. 

\section{A limit theorem}
In this section, we prove a limit theorem using Bagchi's method. 
However, we note that it differs from the original Bagchi' method. 
This section is based on \cite{Ko} and \cite{La96}. 

We use the next Lemma to show Lemma~\ref{mean value} and Lemma~\ref{Prob2}.

\begin{lemma} [cf. {\cite[Lemma~4.2]{Ti}}] \label{adm}
Let $F(x)$ be a real differentiable function such that $F'(x)$ is monotonic, and $F'(x) \ge m > 0$, or $F'(x) \le -m < 0$, throughout the interval $(a, b)$. Then 
\[
\left| \int_{a}^{b} e^{i F(x)}\, dx \right| \le \frac{4}{m}.
\]
\end{lemma}

First, we prove the following mean value theorem of Dirichlet polynomials. 
\begin{lemma} \label{mean value}
Let $\gamma \in \mathcal{F}$. 
Let $\{a_n\}_{n \ge 1}$ be complex numbers that satisfy $a_n \ll_{\delta} n^{\delta}$ for any $\delta > 0$.
Then, 
\[
\int_{T}^{cT} \left| \sum_{n \le \gamma(t)} \frac{a_n}{n^{\sigma + i\gamma(t)}} \right|^2 \, dt 
\ll_{c, \sigma, \varepsilon} T\left(\sum_{n \le \gamma(cT)}\frac{|a_n|^2}{n^{2\sigma}} + \gamma(T)^{1-2\sigma + 2\varepsilon}\log{\gamma(T)}\right)
\]
for any $T \ge T_0$, $c > 1$, $1/2 < \sigma < 1$ and $0 < \varepsilon < \sigma -1/2$, and uniformly for $1/2 < \sigma' \le \sigma \le \sigma'' < 1$.
\end{lemma}

\begin{proof}
First, we calculate the left hand side:
\begin{align*}
&\int_{T}^{cT} \left| \sum_{n \le \gamma(t)} \frac{a_n}{n^{\sigma + it}} \right|^2 \, dt \\
&= \int_{T}^{cT} \sum_{m \le \gamma(t)} \frac{a_m}{m^{\sigma + i\gamma(t)}} \sum_{n \le \gamma(t)} \frac{\overline{a_n}}{n^{\sigma - i\gamma(t)}} \, dt \\
&= \sum_{m, n \le \gamma(cT)} \frac{a_m \overline{a_n}}{(mn)^\sigma} \int_{T_{m, n}}^{cT} \left(\frac{n}{m} \right)^{i\gamma(t)} \, dt \\ 
&= cT\sum_{n \le \gamma(cT)} \frac{|a_n|^2}{n^{2\sigma}} - \sum_{n \le \gamma(cT)} \frac{|a_n|^2}{n^{2\sigma}}T_n \\ 
&\ \ + \sum_{\substack{m, n \le \gamma(cT) \\ m \ne n}}\frac{a_m \overline{a_n}}{(mn)^\sigma} \int_{T_{m, n}}^{\gamma(cT)} \left(\frac{n}{m} \right)^{i\gamma(t)} \, dt, 
\end{align*}
where $T_{m, n} = \gamma^{-1}(m)$ if $m > \gamma(T)$ and $m \ge n$, $T_{m, n} = \gamma^{-1}(n)$ if $n > \gamma(T)$ and $n \ge m$, and $T_{m, n} = T$ if $\gamma(T) \ge m, n$. 

We consider the second term. 
Since $\gamma(\tau)$ is increasing, $\gamma^{-1}(\tau)$ is also increasing. 
Therefore, we have 
\begin{align*}
\sum_{n \le \gamma(cT)} \frac{|a_n|^2}{n^{2\sigma}}T_n 
&= T\sum_{n \le \gamma(T)} \frac{|a_n|^2}{n^{2\sigma}} + \sum_{\gamma(T) < n \le \gamma(cT)}\frac{|a_n|^2}{n^{2\sigma}} \gamma^{-1}(n) \\
&\ll_{\varepsilon} T\sum_{n \le \gamma(T)} \frac{|a_n|^2}{n^{2\sigma}} + \sum_{\gamma(T) < n \le \gamma(cT)}\frac{|a_n|^2}{n^{2\sigma}}\gamma^{-1}(\gamma(cT)) \\
&\ll_{c} T\sum_{n \le \gamma(cT)} \frac{|a_n|^2}{n^{2\sigma}}.
\end{align*}

Next, we consider the third term. 
Using Lemma~\ref{adm}, for $m \ne n$, we have 
\begin{align*}
\int_{T_{m, n}}^{cT} \left(\frac{n}{m} \right)^{i\gamma(t)} \, dt 
&= \int_{T_{m, n}}^{cT} \exp\left(i \gamma(t) \log\frac{n}{m} \right)\, dt  \\
&\ll \frac{1}{\gamma'(T_{m. n}) \left|\log\frac{m}{n}\right|}. 
\end{align*}
Then, we obtain 
\begin{align*}
&\sum_{\substack{m, n \le \gamma(cT) \\ m \ne n}}\frac{a_m \overline{a_n}}{(mn)^\sigma} \int_{T_{m, n}}^{cT} \left(\frac{n}{m} \right)^{i\gamma(t)} \, dt \\
&\ll \sum_{m < n \le \gamma(cT)} \frac{1}{\gamma'(T_{n})(mn)^{\sigma -\varepsilon}\log\frac{m}{n}} \\
&= \frac{1}{\gamma'(T)}\sum_{m < n \le \gamma(T)} \frac{1}{(mn)^{\sigma -\varepsilon}\log\frac{m}{n}} \\
&\ \ \ + \sum_{\substack{m < n \le \gamma(cT) \\ \gamma(T) < n}}\frac{1}{\gamma'(\gamma^{-1}(n))(mn)^{\sigma -\varepsilon}\log{\frac{n}{m}}} \\
&:= \Sigma_1 + \Sigma_2. 
\end{align*}
From \cite[Chapter~VII, Section~7.2, Lemma]{Ti} and (F3) (a), we have  
\begin{align*}
\Sigma_1
\ll_{\sigma, \varepsilon} \frac{1}{\gamma'(T)}\gamma(T)^{2 -2\sigma + 2\varepsilon} \log{\gamma(T)} \ll T\gamma(T)^{1 -2\sigma + 2\varepsilon} \log{\gamma(T)} .
\end{align*}
Finally, we consider $\Sigma_2$. 
We note that $\gamma'(\gamma^{-1}(\tau)) = 1/(\gamma^{-1})'(\tau)$ and $(\gamma^{-1})'(\tau) \ll \gamma^{-1}(\tau)/\tau$ by the inverse function theorem and (F3).  
Since $\gamma^{-1}(\tau)$ is an increasing function, we have 
\begin{align*}
  \Sigma_2 \ll_{c} T\sum_{\substack{m < n \le \gamma(cT) \\ \gamma(T) < n}}\frac{1}{m^{\sigma -\varepsilon}n^{\sigma + 1 -\varepsilon}\log{\frac{n}{m}}}. 
\end{align*}
Using the same way as the estimate of $\Sigma_1$, we conclude 
\[
\Sigma_2 \ll_{\sigma, \varepsilon} T \gamma(T)^{1-2\sigma + 2\varepsilon} \log{\gamma(T)}.
\]
\end{proof}

Now we fix a real-valued smooth function $\varphi(x)$ on $[0, \infty)$ with compact support satisfying $\varphi(x) = 1$ for $0 \le x \le 1$ and $0 \le \varphi(x) \le 1$ for $x \ge 0$. 

We put 
\[
\zeta_X(s) = \sum_{n=1}^{\infty}\frac{\varphi(n/X)}{n^{s}},\ \zeta_X(s, \omega_j) = \sum_{n=1}^{\infty}\frac{\omega_j(n) \varphi(n/X)}{n^{s}}
\]
and
\[
\underline{\zeta}_X(s + i\underline{\gamma}(\tau)) = (\zeta_X(s + i\gamma_1(\tau)), \dots, \zeta_X(s + i\gamma_r(\tau))),
\]
\[
\underline{\zeta}_X(s, \omega) = (\zeta_X(s, \omega_1), \dots, \zeta_X(s, \omega_r))
\]
for $X \ge 2$ and $\omega \in \Omega^r$. 

\begin{lemma} [c.f. {\cite[Theorem~4.11]{Ti}}] \label{approximation}
We have 
\[
\zeta(s) = \sum_{n \le x} \frac{1}{n^{\sigma + it}} + \frac{x^{1-s}}{s-1} + O(x^{-\sigma})
\]
uniformly for $\sigma \ge \sigma' > 0$, $|t| < 2\pi x/C$, where $C$ is a given constant greater than 1.
\end{lemma}

\begin{lemma} \label{Prob1}
For all $\gamma_1, \dots, \gamma_r \in \mathcal{F}$, 
\[
\lim_{X \to \infty} \limsup_{T \to \infty} \frac{1}{T}\int_{T}^{2T} \underline{d}(\underline{\zeta}(s + i\underline{\gamma}(\tau)), \underline{\zeta}_X(s +i\underline{\gamma}(\tau)))\, d\tau = 0.
\]
\end{lemma}

\begin{proof}
From the definition of the metric $\underline{d}$, it suffices to show that

\[
\lim_{X \to \infty} \limsup_{T \to \infty} \frac{1}{T}\int_{T}^{2T} \sup_{s \in C} |\zeta(s + i\gamma(\tau)) - \zeta_X(s +i\gamma(\tau))|\, d\tau = 0
\]
for all compact set $C \subset \mathcal{R}$ and any $\gamma \in \mathcal{F}$. 

Let $T$ be sufficiently large. 
Applying Cauchy's integral formula, we have 
\[
\zeta(s + i\gamma(\tau)) - \zeta_X(s +i\gamma(\tau)) 
= \frac{1}{2\pi i} \int_{\partial{\mathcal{R}}} \frac{\zeta(z + i\gamma(\tau)) - \zeta_X(z +i\gamma(\tau))}{z - s}\, dz
\]
for $s \in C$.
Using this equation and the fact that $\gamma(cT) - \gamma(T) \to \infty$ as $T \to \infty$ for $c > 1$ and $\gamma \in \mathcal{F}$ (this fact can be shown by the mean value theorem), we obtain the following inequalities: 
\begin{align*}
&\frac{1}{T} \int_{T}^{2T} \sup_{s \in C} |\zeta(s + i\gamma(\tau)) - \zeta_X(s +i\gamma(\tau))|\, d\tau \\
&\le \frac{1}{2 \pi T}  \int_{T}^{2T} \sup_{s \in C} \int_{\partial{\mathcal{R}}} \frac{|\zeta(z + i\gamma(\tau)) - \zeta_X(z +i\gamma(\tau))|}{|z - s|}\, |dz| d\tau \\
&\ll_C \frac{1}{T} \int_{\partial{\mathcal{R}}}\, |dz| \int_{T}^{2T} |\zeta(z + i\gamma(\tau)) - \zeta_X(z +i\gamma(\tau))|\, d\tau \\
&\ll_{\mathcal{R}}  \frac{1}{T} \sup_{\sigma: s \in \partial{\mathcal{R}}} \int_{\frac{1}{2}T}^{\frac{5}{2}T} |\zeta(\sigma + i\gamma(t)) - \zeta_X(\sigma +i\gamma(t))|\, dt.
\end{align*}
Applying Lemma~\ref{approximation} with $x = \gamma(t)$, we have 
\begin{align*}
&\int_{\frac{1}{2}T}^{\frac{5}{2}T} |\zeta(\sigma + i\gamma(t)) - \zeta_X(\sigma +i\gamma(t))|\, dt \\
&\le \int_{\frac{1}{2}T}^{\frac{5}{2}T} \left|\sum_{n \le \gamma(t)}\frac{1}{n^{\sigma + i\gamma(t)}} - \zeta_X(\sigma +i\gamma(t))\right|\, dt \\ 
&\ \ \ + O\left( \int_{\frac{1}{2}T}^{\frac{5}{2}T}\gamma(t)^{-\sigma} \, dt\right) \\
&= \int_{\frac{1}{2}T}^{\frac{5}{2}T} \left|\sum_{X < n \le \gamma(t)}\frac{1 - \varphi(n/X)}{n^{\sigma + i\gamma(t)}}\right|\, dt + O(T\gamma(T/2)^{-\frac{1}{2}}).
\end{align*}
By the Cauchy--Schwarz inequality and Lemma~\ref{mean value}, we have 
\begin{align*}
&\int_{\frac{1}{2}T}^{\frac{5}{2}T} \left|\sum_{X < n \le \gamma(t)}\frac{1 - \varphi(n/X)}{n^{\sigma + i\gamma(t)}}\right|\, dt \\
&\le \left(\int_{\frac{1}{2}T}^{\frac{5}{2}T} 1 \, dt \right)^{\frac{1}{2}} \left( \int_{\frac{1}{2}T}^{\frac{5}{2}T} \left|\sum_{X < n \le \gamma(t)}\frac{1 - \varphi(n/X)}{n^{\sigma + i\gamma(t)}}\right|^2 \, dt\right)^{\frac{1}{2}} \\
&\ll_{\sigma_0, \varepsilon} T \left(\sum_{X < n} \frac{1}{n^{2\sigma_0}} 
 + \gamma(T)^{1- 2\sigma_0 + 2\varepsilon} \log{\gamma(T)} \right)^{\frac{1}{2}}.
\end{align*}
Since $\sigma_0 > 1/2$, we conclude 
\[
\limsup_{T \to \infty} \frac{1}{T}\int_{T}^{2T} \sup_{s \in C} |\zeta(s + i\gamma(\tau)) - \zeta_X(s +i\gamma(\tau))|\, d\tau \\
\ll \left(\sum_{X < n} \frac{1}{n^{2\sigma_0}}\right)^{\frac{1}{2}} \to 0
\]
as $X \to \infty$. 
\end{proof} 

Let $\mathcal{P}_0$ be a finite set of prime numbers. 
Let $\Omega_{\mathcal{P}_0} = \prod_{p \in \mathcal{P}_0} S_p$ 
and $\Omega^r_{\mathcal{P}_0} = \Omega_{\mathcal{P}_0 1} \times \dots \times \Omega_{\mathcal{P}_0 r}$, 
where $\Omega_{\mathcal{P}_0 j} = \Omega_{\mathcal{P}_0}$ for all $j = 1, \dots, r$.

We define the probability measure on 
$(\Omega^r_{\mathcal{P}_0}, \mathcal{B}(\Omega^r_{\mathcal{P}_0}))$ by 
\[
Q^{\mathcal{P}_0}_T(A) = \frac{1}{T} \mathrm{meas} \left\{\tau \in [T, 2T] : ((p^{-i\gamma_1(\tau)})_{p \in \mathcal{P}_0} , \dots, (p^{-i\gamma_r(\tau)})_{p \in \mathcal{P}_0}) \in A \right\}, 
\]
for $A \in \mathcal{B}(\Omega^r_{\mathcal{P}_0})$. 
Let $\mathbf{m}^r_{\mathcal{P}_0}$ be the probability Haar measure on $(\Omega^r_{\mathcal{P}_0}, \mathcal{B}(\Omega^r_{\mathcal{P}_0}))$. 
Then, the next lemma holds. 

\begin{lemma} \label{Prob2}
Let $(\gamma_1, \dots, \gamma_r)$ be admissible. 
Then the probability measure $Q^{\mathcal{P}_0}_T$ converges weakly to $\mathbf{m}^r_{\mathcal{P}_0}$ as $T \to \infty$.  
\end{lemma}

\begin{proof}
Let $T$ be sufficiently large. 
We consider the Fourier transformation $g_T(\underline{\mathbf{n}}_1, \dots, \underline{\mathbf{n}}_r)$ of $Q^{\mathcal{P}_0}_T$ for 
$(\underline{\mathbf{n}}_1, \dots, \underline{\mathbf{n}}_r) = ((n_{1 p})_{p \in \mathcal{P}_0}, \dots (n_{r p})_{p \in \mathcal{P}_0}) \in (\mathbb{Z}^{\#{\mathcal{P}_0}})^r$. 
It is well known that 
\[
g_T(\underline{\mathbf{n}}_1, \dots, \underline{\mathbf{n}}_r) := \int_{\Omega^r_{\mathcal{P}_0}} \prod_{j=1}^{r} \prod_{p \in \mathcal{P}_0} x^{n_{j p}}_p \,dQ^{\mathcal{P}_0}_T = \frac{1}{T}\int_{T}^{2T} \prod_{j=1}^{r} \prod_{p \in \mathcal{P}_0} p^{-i n_{j p} \gamma_j(\tau)} \, d\tau.
\]
If $(\underline{\mathbf{n}}_1, \dots, \underline{\mathbf{n}}_r) = (\bar{0}, \dots, \bar{0})$, it is clear that $g_T(\underline{\mathbf{n}}_1, \dots, \underline{\mathbf{n}}_r) = 1$. 
If $(\underline{\mathbf{n}}_1, \dots, \underline{\mathbf{n}}_r) \ne (\bar{0}, \dots, \bar{0})$, then there exists $j = 1, \dots, r$ such that $\sum_{p \in \mathcal{P}_0}n_{j p} \log{p} \ne 0$ since $\{\log{p} : p \in \mathcal{P}_0 \}$ is linearly independent over the field of rational numbers. 
We put $c_j = \sum_{p \in \mathcal{P}_0}n_{j p} \log{p}$ for any $j = 1, \dots, r$. 
Since $(c_1, \dots, c_r) \ne \bar{0}$, we have 
\[
\frac{1}{T}\int_{T}^{2T} \prod_{j=1}^{r} \prod_{p \in \mathcal{P}_0} p^{- i n_{j p} \gamma_j(\tau)} \, d\tau 
\ll_{c_1, \dots, c_r} \frac{1}{T}
\]
by Lemma~\ref{adm} and the definition of admissibility. 
Then 
\begin{align*}
\lim_{T \to \infty} g_T(\underline{\mathbf{n}}_1, \dots, \underline{\mathbf{n}}_r) = 
\begin{cases}
1 & \mbox{if $(\underline{\mathbf{n}}_1, \dots, \underline{\mathbf{n}}_r) = (\bar{0}, \dots, \bar{0})$}, \\
0 & \mbox{if $(\underline{\mathbf{n}}_1, \dots, \underline{\mathbf{n}}_r) \ne(\bar{0}, \dots, \bar{0})$}.
\end{cases}
\end{align*}
This implies the conclusion (cf. \cite[Theorem~3.21]{La96}).

\end{proof}

\begin{lemma} \label{Prob3}
We have 
\[
\lim_{X \to \infty} \mathbb{E}^{\mathbf{m}^r} [\underline{d}(\underline{\zeta}(s, \omega), \underline{\zeta}_X(s, \omega))] = 0. 
\]
\end{lemma}

\begin{proof}

It follows from \cite[Proposition~3.2.11]{Ko}. 

\end{proof}

\begin{proposition} \label{prop1}
The probability measure $\mathcal{Q}_T$ converges weakly to $\mathcal{Q}$ as $T \to \infty$.
\end{proposition}

\begin{proof}
Using the Portmanteau theorem (see \cite[Theorem~13.16]{Kl}), Lemma~\ref{Prob1}, Lemma~\ref{Prob2}, and Lemma~\ref{Prob3}, 
we can prove this proposition. 

\end{proof}

\section{Proof of main theorem}
Finally, we show the main theorem with a property of the support of the measure $\mathcal{Q}$. 
For probability space $(\Omega, \mathcal{M}, \mathbb{P})$, a minimal closed set $S \subset \Omega$ such that $\mathbb{P}(S) = 1$ is called a support of $\mathbb{P}$.  

Let 
\[
S := \{f \in \mathcal{H}(\mathcal{R}) :\text{ $f(s) \ne 0$ or $f(s) \equiv 0$} \}.
\]
Then, from Reference~\cite{La21Rama}, the support of the measure $\mathcal{Q}$ coincides with $S^r$.

\begin{proof}[Proof of Theorem~\ref{main}]

Let $(\gamma_1, \dots, \gamma_r)$ be admissible.
Let $\mathcal{K}$ be a compact subset contained in the strip $1/2 < \sigma < 1$ with connected complement. 
Then, we take $\mathcal{R}$ with (\ref{definition of R}).  
Assume that $f_1(s), \dots, f_r(s)$ are non-vanishing continuous functions on $\mathcal{K}$ and holomorphic on the interior of $\mathcal{K}$. 
Fix $\varepsilon >0$. 
By Mergelyan's theorem, there exist polynomials $p_1(s), \dots, p_r(s)$ such that
\[
\max_{1 \le j \le r} \sup_{s \in \mathcal{K}} |f_j(s) - e^{p_j(s)}| < \varepsilon.
\]
Here we define an open set of $\mathcal{H}^r(\mathcal{R})$ by 
\[
\Phi(p_1, \dots, p_r) := \left\{(g_1(s), \dots, g_r(s)) \in \mathcal{H}^r(\mathcal{R}) : \max_{1 \le j \le r}\sup_{s \in \mathcal{K}} |g_j(s) - e^{p_j(s)}| < \varepsilon \right\}.
\] 
Applying the Portmanteau theorem, Proposition~\ref{prop1} and a property of support of $\mathcal{Q}$, 
we have  
\begin{align*}
&\liminf_{T \to \infty} \frac{1}{T} \mathrm{meas} \left\{\tau \in [T, 2T] :  \max_{1 \le j \le r} \sup_{s \in \mathcal{K}} |\zeta(s +i\gamma_j(\tau)) - e^{p_j(s)}| < \varepsilon \right\} \\
&= \liminf_{T \to \infty} \mathcal{Q}_T(\Phi(p_1, \dots, p_r)) \ge \mathcal{Q}(\Phi(p_1, \dots, p_r)) >0.
\end{align*}
Now, the inequality
\begin{align*}
&\max_{1 \le j \le r}\sup_{s \in \mathcal{K}} |\zeta(s + i\gamma_j(\tau)) -f_j(s)| \\
&\le \max_{1 \le j \le r} \sup_{s \in \mathcal{K}} |\zeta(s + i\gamma_j(\tau)) - e^{p_j(s)}| + \max_{1 \le j \le r} \sup_{s \in \mathcal{K}} |e^{p_j(s)} - f_j(s)|
\end{align*}
holds. 
Thus, we obtain the conclusion. 

\end{proof}

\appendix
\section{Constructing admissible tuples and examples} \label{ap}

In this section, we introduce constructing admissible tuples and examples of admissible tuples. 

The following lemma can be used to construct an admissible tuple. 

\begin{lemma} \label{ad}
A tuple $(\gamma_1, \cdots, \gamma_r)$ is admissible if $(\gamma_1, \cdots, \gamma_r) \in \mathcal{F}^r$ satisfies one of the following assumptions. 

\begin{enumerate}
    \item [(A1)] $r = 1$.

    \item[(A2)] For $r = 2$, then for any $1 \le i < j \le r$, 
    \begin{equation}       
\lim_{\tau \to \infty} \frac{\gamma''_i(\tau)}{\gamma''_j(\tau)}  = c_{i, j} \in \mathbb{R}_{\ge 0}, 
   \label{limit1}
    \end{equation}
    and 
    \begin{align}
\lim_{\tau \to \infty} (\gamma'_i(\tau) - c_{i, j} \gamma'_j(\tau)) = 
\begin{cases}
+\infty& (c_{i, j} = 0)\\
-\infty & (c_{i, j} \ne 0)
\end{cases}
\label{limit2}
 \end{align}
 hold. 
 In particular, the second limit diverges monotonically. 

 \item[(A3)] For $r \ge 3$, the condition (\ref{limit1}) and (\ref{limit2}) hold, and for all $1 \le i < j < k \le r$, 
 \[
\lim_{\tau \to \infty} \frac{\gamma''_i(\tau) -c_{i,k} \gamma''_k(\tau)}{\gamma''_j(\tau) -c_{j,k} \gamma''_k(\tau)} = 0
\]
hold. 
 
\end{enumerate}

\end{lemma}

\begin{proof}
We can prove this lemma with l'Hôpital's rule.
\end{proof}

For example, by Lemma~\ref{ad}, we can see that 
$(\tau^a, e^\tau)\ (a \ge 1)$, 
$(e^{\tau}, e^{2\tau}, \cdots, e^{r\tau})$, 
$(e^{\tau}, e^{\tau^2}, \cdots, e^{\tau^r})$,   
$(e^\tau, e^{\tau^\tau})$, 
$(e^{2\tau}, e^{2\tau} + e^\tau)$, 
$(e^{2\tau}, e^{2\tau} + e^\tau, \Gamma(\tau))$ are  admissible, where $\Gamma(\tau)$ is the gamma function.  
On the other hand, for instance, $(e^\tau, 2e^\tau)$, $(e^\tau, e^{2\tau}, e^{2\tau} + e^\tau)$ are not admissible according to the definition of admissibility.

We note that if $\gamma \in \mathcal{F}$, then the divergence speed of $\gamma$ is faster than that of some polynomials, as stated in (F3). 
Therefore, examples of shifts considered in \cite{LMS19} -- \cite{LS} and some case of \cite{Pa18} belong to $\mathcal{F}$. 
On the other hand, in Reference~\cite{An}, we also consider shifts as exponential shifts and so on. 
However, for example, $\gamma(\tau) = \exp(\exp(\exp(\tau)))$ does not satisfy their assumptions. 
But $\gamma(\tau) = \exp(\exp(\exp(\tau)))$ belongs to $\mathcal{F}$. 
Therefore, we can consider faster divergence functions than Reference~\cite{An}. 
Finally, we need assumptions (F1) -- (F3) in the proof of Lemma~\ref{mean value} and Lemma~\ref{Prob1}. 

\subsection*{Acknowledgments}
The author would like to thank Professor Sh\={o}ta Inoue and Professor Kohji Matsumoto for their helpful comments. 
This work was financially supported by JST SPRING, Grant Number JPMJSP2125.




\begin{flushleft}
{\footnotesize
{\sc
Graduate School of Mathematics, Nagoya University, Chikusa-ku, Nagoya 464-8602, Japan.
}\\
{\it E-mail address}: {\tt m21029d@math.nagoya-u.ac.jp}
}
\end{flushleft}

\end{document}